\documentclass[11pt]{article}
\usepackage{a4}
\usepackage{amsmath}
\usepackage{amssymb}
\usepackage{amsthm}
\usepackage{graphicx}
\usepackage{amsfonts}

\newtheorem{theorem}{Theorem}[section]

\newtheorem{lemma}[theorem]{Lemma}
\theoremstyle{remark}
\newtheorem{remark}[theorem]{Remark}

\newcommand{\R}{\mathbb{R}}

\newcommand{\K}{{\mathsf K}}
\newcommand{\KF}{{\mathcal K}}
\newcommand{\NK}{{\mathsf N}}

\renewcommand{\L}{{\mathsf L}}

\newcommand{\C}{\mathcal{C}}
\newcommand{\G}{\mathsf{G}}

\DeclareMathOperator{\sd}{sd}

\title{$d$-representability of simplicial complexes of fixed dimension\thanks{
I have obtained the main result of this note when I was working on my PhD
thesis. Thus the contents of this contribution also appears in modified version
in my PhD thesis.
}
}
\author{Martin Tancer\thanks{Department of Applied Mathematics and Institute for Theoretical Computer
Science (supported by project 1M0545
of The Ministry of Education of the Czech Republic), Faculty of Mathematics
and Physics, Charles University, Malostransk\'e n\'am.~25, 118~00 Prague,
Czech Republic. Partially supported by project GAUK 421511. {\tt Email:tancer@kam.mff.cuni.cz}.
}}

\begin{document}
\maketitle
\begin{abstract}
Let $\K$ be a simplicial complex with vertex set $V = \{v_1, \dots, v_n\}$. The
complex $\K$ is $d$-representable if there is a collection $\{C_1, \dots,
C_n\}$ of convex sets in $\R^d$ such that a subcollection $\{C_{i_1}, \dots,
C_{i_j}\}$ has a nonempty intersection if and only if $\{v_{i_1}, \dots,
v_{i_j}\}$ is a face of $\K$.

In 1967 Wegner proved that every simplicial complex of dimension $d$ is $(2d +
1)$-representable. He also suggested that his bound is the best possible, i.e.,
that there are $d$-dimensional simplicial complexes which are not
$2d$-representable. However, he was not able to prove his suggestion.

We prove that his suggestion was indeed right. Thus we add another piece to the
puzzle of intersection patterns of convex sets in Euclidean space.
\end{abstract}

\section{Introduction}
Let $\C$ be a collection of sets. The \emph{nerve} of $\C$ is a simplicial
complex\footnote{We assume that the reader is familiar with simplicial
complexes; otherwise we refer him to standard sources such
as~\cite{hatcher01,munkres84,matousek03}.} with vertex set $\C$ and with faces of the form $\{C_1, \cdots, C_k\}
\subseteq \C$ such that the intersection $C_1 \cap \cdots \cap C_k$ is
nonempty. We say that a simplicial complex is \emph{$d$-representable} 
if it is isomorphic to the nerve of a finite collection of convex sets in
$\R^d$. This notion is designed to capture possible `intersection patterns' of
convex sets in $\R^d$. Study of intersection patterns of convex sets is active
since a theorem by Helly~\cite{helly23}. 

Let us also mention that $d$-representable simplicial complexes are very closely
related to well studied intersection graphs of convex sets. An intersection
graph only records which pairs of convex sets have a nonempty intersection;
however, it does not take care of multiple intersections. Thus
$d$-representable complexes provide more detailed information about the
intersection pattern. 

From another point of view, need of understanding intersection patterns of
convex sets appears, e.g., also in manifold learning. The task might be to
reconstruct the homotopy type of a manifold $M$ given by sample points $\{p_i\}$.
Sample points can be enlarged to convex sets $\{C_i\}$; and under certain
conditions $M$ is homotopic to $\bigcup C_i$. On the other, via the nerve
theorem, $\bigcup C_i$ is homotopic to the nerve of $\{C_i\}$. See,
e.g.,~\cite{attali-lieutier10} for more details.

The reader is referred to~\cite{eckhoff85} or~\cite{tancer11surveyarxiv} for
more background on intersection patterns of convex sets. 

One of the question arising in this area is how the dimension of a complex
affects $d$-representability. Wegner~\cite{wegner67} showed that a complex of
dimension $d$ is always $(2d+1)$-representable. (This result was also
independently found by Perel'man~\cite{perelman85}.) Wegner also suggested that
the value $2d+1$ is the best possible, i.e., that there are $d$-dimensional
simplicial complexes which are not $2d$-representable. (The question about the
best possible value is also reproduced by Eckhoff~\cite{eckhoff85}, and the
author is not aware that this question would be answered yet.) 

Wegner proved that the
barycentric subdivision\footnote{In this case, every edge is subdivided into two
edges and a new vertex in the center of the edge is inserted.} of a nonplanar
graph is not $2$-representable. He also suggested that the barycentric
subdivision of a $d$-dimensional complex that does not embed into $\R^{2d}$ is
not $2d$-representable; however, he was not able to prove his suggestion. 

In this short note we prove that the value $2d + 1$ is indeed the best
possible. Let $\Delta_n$ denotes the full simplex of dimension $n$ and let
$\K^{(k)}$ denotes the $k$-skeleton of a simplicial complex $\K$. We prove that
the barycentric subdivision of $\Delta^{(d)}_{2d+2}$ and also the barycentric
subdivision of many other complexes is not $d$-representable; see the precise
statement below.

\begin{theorem}
\label{t:dimen}
The barycentric subdivision of $\Delta^{(d)}_{2d+2}$ is not $d$-representable.
More generally, if $\L$ is a $d$-dimensional simplicial complex with vanishing
Van Kampen obstruction, then the barycentric subdivision $\sd \L$ is not
$d$-representable.
\end{theorem}

\begin{remark}
\emph{Van Kampen obstruction} is a certain cohomology obstruction for embeddability
$d$-dimensional simplicial complexes into $\R^{2d}$. We are not going to define
this obstruction precisely since we would need to many preliminaries. The
interested reader is referred either to~\cite{melikhov09} for a survey or
to~\cite[Appendix D]{matousek-tancer-wagner11} for an elementary exposition.

Let us just mention some properties of Van Kampen obstruction. If $\K$ is a
$d$-dimensional simplicial complex which embeds into $\R^{2d}$, then its Van
Kampen obstruction has to vanish. If $d \neq 2$, then also the converse is true,
i.e., a $d$-dimensional simplicial complex with vanishing Van Kampen
obstruction embeds into $\R^{2d}$. In case $d = 2$ there are, however,
simplicial complexes with vanishing Van Kampen obstruction which do not embed
into $\R^4$; see~\cite{freedman-krushkal-teichner94}.
\end{remark}

Regarding our proof method, let us first indicate Wegner's approach for case
$d=1$. Let $\G$ be a nonplanar graph (graph is a $1$-dimensional simplicial
complex). Assuming that $\sd \G$ was $2$-representable, Wegner is able to
construct a piecewise linear embedding $g$ of the geometric realization $|\sd
\G|$ into $\R^2$. This contradicts the fact that $\G$ is nonplanar.

It seems hard to extend this construction in such a way that $g$ would be an
embedding in higher dimensions. Our main observation is that it is not
necessary to require that $g$ is an embedding in order to obtain a
contradiction with an embeddability-type result. We only construct such a $g$
that disjoint simplices have disjoint images, which is still in contradiction
with vanishing Van Kampen obstruction.

\section{Barycentric subdivision}
In order to set up notation, we recall the definition of a barycentric
subdivision of a simplicial complex.

From geometric point of view we put a new vertex into the barycenter of
every geometric face of a simplicial complex $\K$. 
Then we form a new simplicial complex whose
vertices are the barycenters and whose faces are simplices formed in between
these barycenters.

It is perhaps more convenient to state the precise definition in abstract
setting. Given a simplicial complex $\K$ the \emph{barycentric subdivision} of
$\K$ is a simplicial complex $\sd \K$ whose set of vertices is the set $\K
\setminus \emptyset$ and whose faces are collections $\{\alpha_1, \dots,
\alpha_m\}$ of faces of $\K$ such that
$$
\alpha_1 \supsetneq \alpha_2 \supsetneq \cdots \supsetneq \alpha_m \neq
\emptyset.
$$
The vertices of $\sd \K$ play role of barycenters of faces of $\K \setminus
\emptyset$. The faces of $\sd \K$ are the simplices in between of these
barycenters. See Figure~\ref{f:baryc}.

\begin{figure}
\begin{center}
\includegraphics{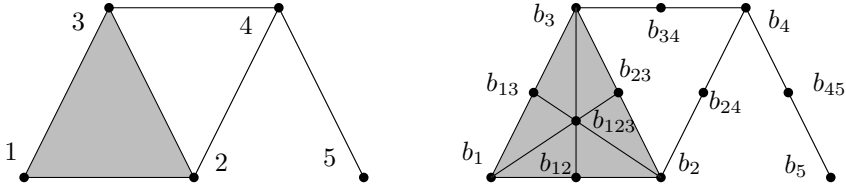}
\end{center}
\caption{Barycentric subdivision of a complex. For example, the vertex $b_{13}$
denotes the barycenter of the face $13 = \{1, 3\}$ (in geometric setting).}
\label{f:baryc}
\end{figure}

The complexes $\K$ and $\sd \K$ have the same geometric realization, i.e., $|
\K
| = |\sd \K|$.

\section{Proof}

For the proof we will need two auxiliary results.

\begin{theorem}[Van Kampen - Flores theorem; see,
e.g.,~{\cite[Theorem 5.1.1]{matousek03}}]
\label{t:vkf}
Let $\K = \Delta^{(d)}_{2d+2}$. Then for any continuous map $f\colon |\K|
\rightarrow \R^{2d}$ there are two disjoint $d$-dimensional simplices $\gamma$
and $\delta$ of $\K$ such that their images $f(|\gamma|)$ and $f(|\delta|)$
intersect.
\end{theorem}
We remark that the conclusion of the theorem remains true if $\K$ is replaced
with any $d$-dimensional complex with non-zero Van Kampen obstruction (in
particular, $\K$ has a non-zero Van Kampen obstruction). The fact that
Theorem~\ref{t:vkf} extends to complexes with non-zero obstruction just follows
from one of possible definitions of Van Kampen obstruction (and is trivial for
a reader familiar with this topic); see, e.g., exposition
in~\cite{freedman-krushkal-teichner94}.\footnote{There is a sign error
in~\cite{freedman-krushkal-teichner94} in the definition of Van Kampen
obstruction observed by Melikhov~\cite{melikhov09}. However, it does not affect
our conclusion.} On the other hand,
Theorem~\ref{t:vkf} for our specific $\K$ can be proved on more elementary
level using Borsuk-Ulam theorem; and that is why we also emphasize this
specific case.   


Let $\alpha$ and $\beta$ be faces of a simplicial complex $\K$. We say that
$\alpha$ and $\beta$ are \emph{remote} if there is no edge $ab \in \K$ with $a
\in \alpha, b \in \beta$.

\begin{lemma}
\label{l:remote}
Let $\KF$ be a collection of convex sets in $\R^m$ and let $\K :=
\NK(\KF)$ be the nerve of $\KF$. Then there is a
linear map $g: |\sd \K| \rightarrow \R^m$ such that $g(|\sd \alpha|)
\cap g(|\sd \beta|) = \emptyset$ for any remote $\alpha, \beta \in \K$.
\end{lemma}

\begin{figure}
\begin{center}
\includegraphics{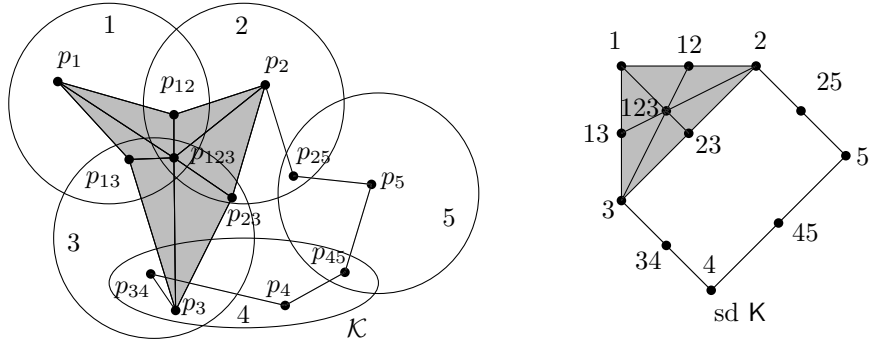}
\caption{Mapping $\sd \K$ into $\KF$. The notation is simplified. For instance
$12$ stands for $\{1, 2\}$, $p_{123}$ stands for $p(\{1,2,3\})$, etc.}
\label{f:mapsd}
\end{center}
\end{figure}

\begin{proof}
First we specify $g$ on the vertices of $\sd \K$ then we extend it linearly to
the whole $\sd \K$. See Figure~\ref{f:mapsd}.

A vertex of $\sd \K$ is a simplex of $\K$, i.e., a subcollection $\KF'$ of
$\KF$
with a nonempty intersection. Let us pick a point $p(\KF')$ inside $\cap \KF'$.
We set $g(\KF') := p(\KF')$
for $\KF' \in \K$. As we already mentioned, we extend $g$ linearly to $\sd
\K$.

If $\alpha = \KF' \in \K$, then $g(|\sd \alpha|) \subseteq \cup \KF'$. Thus
$g(|\sd \alpha|) \cap g(|\sd \beta|) = \emptyset$ for remote $\alpha, \beta \in
\K$.

\end{proof}

\begin{proof}[Proof of Theorem~\ref{t:dimen}]
First we prove the specific case.

Let $\K = \sd \Delta^{(d)}_{2d+2}$. For contradiction we assume that $\K$ is
$2d$-representable. Let $\KF$ be the $2d$-representation of it. (Without loss
of generality $\KF = \NK(\K)$.) According to Lemma~\ref{l:remote} there is a
map $g\colon |\sd \K| \rightarrow \R^{2d}$ such that $g(|\sd \alpha|)
\cap g(|\sd \beta|) = \emptyset$ for any remote $\alpha, \beta \in \K$.

Since $\sd \K = \sd \sd \Delta^{(d)}_{2d+2}$, we have $|\Delta^{(d)}_{2d+2}| =
|\K| = |\sd \K|$, and thus we can also apply $g$ to simplices of
$\Delta^{(d)}_{2d+2}$.

Let $\gamma$ and $\delta$ be disjoint simplices of $\Delta^{(d)}_{2d+2}$. Let
$\alpha$ be a simplex of $\sd \gamma$ and $\beta$ a simplex of $\sd \delta$.
Then $\alpha$ and $\beta$ are remote in $\K$. Thus $g(|\sd \alpha|) \cap g(|\sd
\beta|) = \emptyset$. Consequently, $g(|\gamma|) \cap g(|\delta|) = \emptyset$
for any choice of $\gamma$ and $\delta$. However, this contradicts the Van
Kampen-Flores theorem.

More general part of the theorem is obtained along the same lines when a
generalized version of Theorem~\ref{t:vkf} is used.
\end{proof}

\section*{Acknowledgment}
I would like to thank Xavier Goaoc for point me out that intersection patterns
of convex sets relate to manifold learning.

\bibliographystyle{alpha}
\bibliography{/home/martin/clanky/bib/general}

\end{document}